\documentclass[11pt]{article}
\usepackage{amssymb}
\usepackage{amsthm}
\usepackage{amsmath}
\usepackage{amscd}
\newtheorem{theorem}{Theorem}[section]
\newtheorem{corollary}[theorem]{Corollary}
\newtheorem{lemma}[theorem]{Lemma}
\newtheorem{proposition}[theorem]{Proposition}
\theoremstyle{definition}
\newtheorem{definition}[theorem]{Definition}
\newtheorem{remark}[theorem]{Remark}
\newtheorem{example}[theorem]{Example}

\title{\bf Some Noncommutative Constructions and Their Associated NCCW Complexes }
\author{Vida Milani \\ Ali Asghar Rezaei\\
{\small Faculty of Mathematical Sciences, Shahid Beheshti
University, Tehran, Iran}}
\date{}
\begin{document}
\maketitle

\begin{abstract}
 In this article some  noncommutative topological objects
such as NC mapping cone and NC mapping cylinder are introduced. We
will see that these objects are equipped with the NCCW complex
structure of [P]. As a generalization we introduce the notions of
NC mapping cylindrical and conical telescope. Their relations with
NC mappings cone and cylinder are studied. Some results on their
$K_0$ and $K_1$ groups are obtained and the cyclic six term exact
sequence theorem for their k-groups are proved. Finally we explain
their NCCW coplex structure and the conditions in which these
objects admit NCCW complex structures.
\end{abstract}

\section{Introduction}
In topology, the mapping cylinder and mapping cone for a
continuous map $f:X\rightarrow Y$ between topological spaces are
defined by quotient. These two constructions together the cone and
suspension for a topological space, are important concepts in
classical algebraic topology (especially homotopy theory and CW
complexes).\\
The analog versions of the above constructions in noncommutative
case, are defined for C*-morphisms and C*-algebras. We review this
concepts from [W], and study some results about their related NCCW
complex structure, the notion which was introduced by Pedersen in
[P].\\
Another construction which is studied in algebraic topology, is
the mapping telescope for a $X_1\xrightarrow{f_1}
X_2\xrightarrow{f_2}\dotsb$ of continuous maps between topological
spaces. We define two noncommutative version of this: NC mapping
cylinderical and conical telescope.\\
This paper contains five sections. In sections 2 and 3 we study
NCCW complexes and simplicial morphisms, and prove some results
which we need in the other sections. In section 4 we discus NC
mapping cone and cylinder. Finally in section 5 we introduce the
NC mapping cylinderical and conical telescope.
\section{NCCW Complexes}
 In this section we
explain the notions of NCCW complexes from [P]. To this  regard we
express some basic definitions from [W].
\begin{definition}
 Let $A$ and $C$ be two C*-algebras. An {\it extension} for $A$ with
respect to $C$ is a C*-algebra $B$ together with two morphisms
$\alpha$ and $\beta$ for which the following sequence is exact
$$0\xrightarrow[{\quad}]{} A\xrightarrow[{\quad} ]{\alpha} B\xrightarrow[{\quad}]{\beta} C\xrightarrow[{\quad}]{} 0.$$
\end{definition}
\begin{definition}
A {\it pullback} for the C*-algebra $C$ via C*-morphisms
$\alpha_1:A_1\rightarrow C$ and $\alpha_2:A_2\rightarrow C$ is the
C*-subalgebra $PB$ of $A_1\oplus A_2$ defined by
$$PB:=\{a_1\oplus a_2\in A_1\oplus A_2 |\alpha_1(a_1)=\alpha_2(a_2)\}$$
\end{definition}
From now on the pullback decomposition notation $PB:=A_1\bigoplus\limits_{C}A_2$ is  used all throughout this paper.\\
\begin{remark} Since $\alpha_1$ and $\alpha_2$ are continuous maps.
$PB$ is closed in $A_1\oplus A_2$ and so it is a C*-algebra.
\end{remark}
\begin{remark}
 From the above definition it follows that the pullback
satisfies the following universality property; i.e. \\
i)  It
commutes the following diagram
$$\begin {CD}PB@>{\pi_2}>>
A_2\\@V{\pi_1}VV @VV{\alpha_2}V\\ A_1@>{\alpha_1}>>C \end{CD}$$
 ($\pi_1$
and $\pi_2$ are projections onto the first and second
coordinates)\\
ii) For any C*-algebra $D$ and any two C*-morphisms
$\delta_1:D\rightarrow A_1$ and $\delta_2:D\rightarrow A_2$ for
which the following diagram commutes, $$\begin{CD} D@>{\delta_2}>>
A_2\\@V{\delta_1}VV @VV{\alpha_2}V\\ A_1@>{\alpha_1}>>C\end{CD}$$
there exists a unique C*-morphism $\Delta :D\rightarrow PB$ such
that the following diagram commutes.

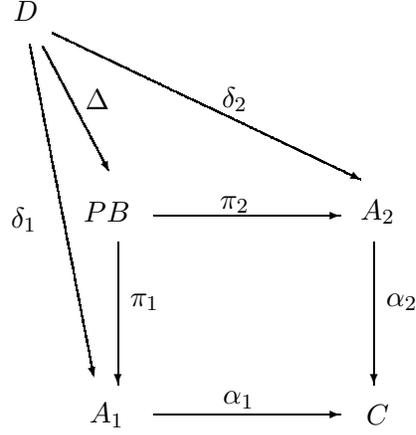
\unitlength 1.00mm 
\linethickness{0.4pt}
\ifx\plotpoint\undefined\newsavebox{\plotpoint}\fi 
\begin{picture}(105.25,60.75)(20,90)
\put(97,118.5){\vector(1,0){.07}}\put(72.25,118.5){\line(1,0){24.75}}
\put(97,92){\vector(1,0){.07}} \put(72.25,92){\line(1,0){24.75}} 

\put(67.5,96){\vector(0,-1){.07}} \put(67.5,115){\line(0,-1){19}}


\put(101.5,96){\vector(0,-1){.07}}\put(101.5,115){\line(0,-1){19}}

\put(66.25,124.5){\vector(1,-2){.07}}
\multiput(57.5,141)(.033653846,-.063461538){260}{\line(0,-1){.063461538}}
\put(66,92){\makebox(0,0)[cc]{$A_1$}}
\put(102,92){\makebox(0,0)[cc]{$C$}}
\put(66,119){\makebox(0,0)[cc]{$PB$}}
\put(102,119){\makebox(0,0)[cc]{$A_2$}}
\put(83.5,94){\makebox(0,0)[cc]{$\alpha_1$}}
\put(83,120){\makebox(0,0)[cc]{$\pi_2$}}
\put(105.25,107){\makebox(0,0)[cc]{$\alpha_2$}}
\put(71,107){\makebox(0,0)[cc]{$\pi_1$}}
\put(64.75,134){\makebox(0,0)[cc]{$\Delta$}}
\put(83,134){\makebox(0,0)[cc]{$\delta_2$}}
\put(55,118){\makebox(0,0)[cc]{$\delta_1$}}
\put(64.25,97.25){\vector(1,-4){.07}}
\multiput(55.75,141.25)(.033730159,-.174603175){252}{\line(0,-1){.174603175}}
\put(99.75,123.25){\vector(2,-1){.07}}
\multiput(58.75,142.75)(.070934256,-.033737024){578}{\line(1,0){.070934256}}
\put(55.25,145.75){\makebox(0,0)[cc]{$D$}}
\end{picture}

\end{remark}
 In what follows we refer to [P].\\
{\bf Notations.} Set $\mathbb{I}=[0,1]$, $\mathbb{I}^n=[0,1]^n$,
$\mathbb{I}_0^n=(0,1)^n$ and for any C*-algebra $A$,
\begin{alignat*}{2}
\mathbb{I}A=C(\mathbb{I}\rightarrow A)\quad ,\quad &\mathbb{I}^nA&=C(\mathbb{I}^n\rightarrow A)\\
\mathbb{I}_0^nA=C_0(\mathbb{I}_0^n\rightarrow A)\quad,\quad
&S^nA&=C(S^n\rightarrow A)\end{alignat*} Here
$\mathbb{I}^{n+1}/\mathbb{I}_0^{n+1}$ is identified whit the
sphere $S^n$.

 All the above sets together with the usual pointwise
operations, and supremum norm  are C*-algebras.
\begin{definition}
The {\it NCCW complexes} are defined by induction on their
dimension as follows.\\ A NCCW complex of dimension zero is
defined to be a finite linear dimensional C*-algebra
 $A_0$ corresponding to the decomposition $A_0=\bigoplus\limits_{k}M_{n(k)}$ of finite dimensional matrix algebras.\\
In dimension $n$, a NCCW complex is defined as a sequence of
C*-algebras $\{A_0,A_1,\dots A_n\}$, where each $A_k$ obtained
inductively from the previous one by the following pullback
construction. $$\begin{CD} 0@>>>\mathbb{I}_0^kF_k@>>> A_k
@>{\pi}>> A_{k-1} @>>> 0\\
&&@|@VV{\rho_k}V@VV{\sigma_k}V\\
0@>>>\mathbb{I}_0^kF_k@>>>\mathbb{I}_0^kF_k @>{\partial}>>
S^{k-1}F_k @>>> 0\end{CD}$$ In the above diagram, the rows are
extensions,$F_k$ is some C*-algebra of finite dimension, $\delta$
-the boundary map- is the restriction morphism, $\sigma_k$ the
connecting morphism can be any morphism, and finally, $\rho_k$ and
$\pi$ are projections onto the first and second factors in the
pullback decomposition
$$A_k=\mathbb{I}_0^kF_k\bigoplus\limits_{S^{k-1}F_k}A_{k-1}.$$
\end{definition}
\begin{remark} From the above recursive definition it
follows that for each $n$-dimensional NCCW complex $A_n$, there
corresponds a decreasing family of closed ideals, called canonical
ideals,
$$A_n=I_0\supset I_1\supset\dots\supset I_{n-1}\supset I_n\not=0$$
where $I_n=\mathbb{I}_0^nF_n$ and for each $k\geq 1$,
$I_k/I_{k+1}=\mathbb{I}_0^kF_k$.  Moreover for each $0\leq k\leq
n-1$, $A_n/I_{k+1}$ is a $k$-dimensional NCCW complex.\end{remark}
\begin{example}\label{3} $C(\mathbb{I})$ is a
$1$-dimensional NCCW complexes. For see this, let $F_1=\Bbb C$,
$A_0=\mathbb{C }\oplus\mathbb{C}$ and $A_1=C(\mathbb{I})$, then we
will have $\mathbb{I}_0^1F_1=C_0((0,1))$,
$\mathbb{I}^1F_1=C(\mathbb{I})$, and
$S^0F_1=\mathbb{C}\oplus\mathbb{C}$. And the pullback construction
diagram becomes
$$\begin{CD} 0@>>>C_0((0,1))@>>> C(\mathbb{I})
@>{\pi}>> \mathbb{C }\oplus\mathbb{C}@>>> 0\\
&&@|@VV{\rho_1}V@VVV\\
0@>>>C_0((0,1)@>>>C(\mathbb{I})@>{\partial}>> \mathbb{C
}\oplus\mathbb{C}@>>> 0\end{CD}$$ where $C(\mathbb{I})$ is
identified whit
\begin{align*}
C(\mathbb{I})&\simeq\{f\oplus (\lambda\oplus\mu )\in C(\mathbb{I})\oplus (\mathbb{C}\oplus\mathbb{C})\quad |\quad f(0)=\lambda , f(1)=\mu\}\\
&=C(\mathbb{I})\bigoplus\limits_{\mathbb{C }\oplus\mathbb{C
}}(\mathbb{C}\oplus\mathbb{C})
\end{align*}
And for each $f\in C(I)$ and $\lambda ,\mu\in\mathbb{C}$,
\begin{align*}
\pi (f\oplus(\lambda\oplus\mu ))&=\lambda\oplus\mu\\
\rho_1(f\oplus(\lambda\oplus\mu ))&=f\\
\partial f&=f(0)\oplus f(1)\end{align*} Now the sequence $\{A_0=\mathbb{C }\oplus\mathbb{C }
, A_1=C(I)\}$ makes $C(\mathbb{I})$ into a $1$-dimensional NCCW
complex. The canonical ideals for $C(\mathbb{I})$ are
$$C(\mathbb{I})=I_0\supset I_1=\mathbb{I}_0^1F_1=C_0((0,1))$$
In a similar way we can see that both $C_0((0,1])$ and
$C_0((0,1))$ are  $1$-dimensional NCCW complexes. In \ref{1} we
describe it in another way. \end{example}
\section{Simplicial morphisms}

Simplicial morphisms are the most important morphisms  in the
category of NCCW complexes. They are in fact the NC analogue of
simplicial map on CW complex.
\begin{definition}
A {\it simplicial morphism} from the $n$-dimensional NCCW complex
$A_n$ into the $m$-dimensional NCCW complex $B_m$ is a mapping
$\alpha :A_n\rightarrow B_m$ satisfying the following two
conditions:\\
i) If  \begin{align*}A_n&=I_0\supset I_1\supset\dots\supset I_{n-1}\supset I_n\not=0\\
B_m&=J_0\supset J_1\supset\dots\supset J_{m-1}\supset J_n\not=0
\end{align*}
be the sequences of canonical ideals for $A_n$ and $B_m$, then
$\alpha(I_k)\subset J_k$ for all $k$. Particularly $\alpha(I_k)=0$
for $k>m$.\\
ii) for $0\leq k\leq n$, if $I_k/I_{k-1}=\mathbb{I}_0^kF_k$,
$J_k/J_{k-1}=\mathbb{I}_0^kG_k$ and $\tilde{\alpha_k} :
\mathbb{I}_0^kF_k\longrightarrow\mathbb{I}_0^kG_k$ be the
homomorphism induced by $\alpha$, then there exists a morphism
$\varphi_k:F_k\longrightarrow G_k$ and a homeomorphism $i_k$ of
$\mathbb{I}^k$ such that $\tilde{\alpha_k}=i_k^*\otimes\varphi_k$,
where $i_k^*:C_0(\mathbb{I}^k)\longrightarrow C_0(\mathbb{I}^k)$
is induced by $i_k$. Here $\mathbb{I}_0^kF_k$ is identified with
$C_0(\mathbb{I}^k)\otimes F_k$ and the same for
$\mathbb{I}_0^kG_k$.
\end{definition}
Here we state some properties of simplicial morphisms from [P].(In the following propositions $A_n$ and $B_m$ are NCCW complex of dimensions $n$,$m$):\\
{\bf Proposition (i).}The kernel and the image of a simplicial
morphism are NCCW complexes.\\
{\bf Proposition (ii).}The pullback of an NCCW complex $C$ via
simplicial morphisms $\alpha:A_n\rightarrow C$ and $\beta:B_m\rightarrow C$ is an NCCW complex of dimension $max\{n,m\}$.\\
{\bf Proposition (iii).}The tensor product $A_n\otimes B_m$ of
NCCW complexes is again an NCCW complex of dimension $n+m$.
\begin{example}\label{1}
We show that $C_0((0,1])$, $C_0([0,1))$, and $C_0((0,1))$ are
$1$-dimensional NCCW complexes. By example \ref{3} $C(\mathbb{I})$
is an NCCW complex of dimension $1$. Now $C_0((0,1])$ is the
kernel of the map
\begin{align*}
\alpha :&C(\mathbb{I})\longrightarrow \mathbb{C}\\
&\quad f\longmapsto f(0)\end{align*} which is a siplicial
morphism. (We note that since $\mathbb{C}$ is a zero dimensional
NCCW complex, with the only nonzero ideal $\mathbb{C}$, so
$\alpha$ satisfies the two conditions of being a simplicial
morphism.)So $C_0((0,1])$ is an NCCW complex. Its dimension is
one, because it is not of finite linear dimension(and so it is not
a $0$-dimensional NCCW complex ). Also $C_0([0,1))$ being
identical to $C_0((0,1])$ is an NCCW complex of dimension one. In
a similar way, $C_0((0,1))$ as the kernel of the simplicial
morphism
\begin{align*}
\beta :&C_0((0,1])\longrightarrow \mathbb{C}\\
&\qquad f\longmapsto f(1)\end{align*} is an $1$-dimensional NCCW
complex.
\end{example}
\begin{lemma}
For each NCCW complex $A$ of dimension n, $\mathbb{I}A$ is an NCCW
complex of dimension $n+1$.
\end{lemma}
\begin{proof} It is obvious from the proposition (iii) and the fact that
$\mathbb{I}A$ can be identified with the tensor product
$C(\mathbb{I})\otimes A$.\end{proof}
\begin{proposition}
For any NCCW complex $A$, the evaluation map $ev(1)
:\mathbb{I}A\rightarrow A$ defined by $f\mapsto f(1)$ is a
simplicial morphism.
\end{proposition}
\begin{proof}
Since both $\mathbb{I}A$ and $A$ are NCCW complexes, and $A$ is
embedded in $\mathbb{I}A$, so we can regard $A$ as an NCCW
subcomplex of $\mathbb{I}A$. Now by [P, 11.14] the quotient map
$\mathbb{I}A\rightarrow A$ is a simplicial morphism. So $ev(1)$
(and also $ev(t_0)$ for each $t_0\in\mathbb{I}$) is simplicial,
since it is a quotient map.
\end{proof}
\section{NC constructions}
Following [W] in this section the notions of cone and suspension
for an arbitrary C*-algebra and also the NC mapping cone and NC
mapping cylinder for C*-morphisms are reviewed. We also study
their associated NCCW complexes.
\begin{definition} For
a C*-algebra $A$, the {\it NC cone} over $A$, $CA$, and the {\it
NC suspension} of $A$, $SA$ are defined respectively by
\begin{align*}
CA&:=\{f\in\mathbb{I}A|f(0)=0\}\\
SA&:=\{f\in\mathbb{I}A|f(0)=f(1)=0\}\end{align*}\end{definition}
\begin{remark}From the above definition we have the inclusions
$SA\subset CA\subset\mathbb{I}A$. The following facts are stated
from
[W]:\\ i) $CA$ is contractible.\\
ii) $SA$ is contractible only if $A$ is contractible
\end{remark}
\begin{proposition}
For every NCCW complex $A$ of dimension $n$, the cone $CA$ and the
suspension $SA$ are NCCW complexes of dimension $n+1$.
\end{proposition}
\begin{proof}
This follows from example \ref{1}, proposition(iii), and the
following identifications:
\begin{align*}
CA&=C_0((0,1]\rightarrow A)\simeq C_0(0,1])\otimes A\\
SA&=C_0((0,1)\rightarrow A)\simeq C_0((0,1))\otimes A.\end{align*}
\end{proof}
\begin{definition}
For a C*-morphism $\alpha :A\rightarrow B$ the {\it NC mapping
cone} is defined by $$Cone(\alpha):=\{a\oplus f\in A\oplus CB |
f(1)=\alpha (a)\}$$
\end{definition}
\begin{remark}
$Cone(\alpha )$ satisfies the following pullback diagram
$$\begin {CD}Cone(\alpha)@>{\pi_2}>>
CB\\@V{\pi_1}VV @VV{ev(1)}V\\ A@>{\alpha}>>B \end{CD}$$ where
$\pi_1$ and $\pi_2$ are projections onto the first and second
coordinates and $ev(1)$ is the evaluation map. \end{remark}
\begin{proposition}\label{2}
If $\alpha :A_n\rightarrow B_m$ is a simplicial morphism between
NCCW complexes of dimension $n$ and $m$, then $Cone(\alpha)$ is an
NCCW complex of dimension $max\{n,m+1\}$.
\end{proposition}
\begin{proof}
Since $CB_m$ is an NCCW complex of dimension $m+1$, and $ev(1)$
and $\alpha$ are simplicial morphisms in the pullback diagram for
$Cone(\alpha)$, from proposition (ii) it follows that
$Cone(\alpha)$ is an NCCW complex of dimension $max\{n,m+1\}$.
\end{proof}
\begin{definition}
For a C*-morphism $\alpha :A\rightarrow B$, the {\it NC mapping
cylinder} is defined by $$Cyl(\alpha):=\{a\oplus f\in A\oplus
\mathbb{I}B | f(1)=\alpha (a)\}.$$
\end{definition}
As in $Cone(\alpha)$ for $Cyl(\alpha)$ we have the following
pullback diagram:
$$\begin {CD}Cyl(\alpha)@>{\pi_2}>>
\mathbb{I}B\\@V{\pi_1}VV @VV{ev(1)}V\\ A@>{\alpha}>>B \end{CD}$$
\begin{remark}
Since $CB\subset\mathbb{I}B$, so for any C*-morphism $\alpha$ we
have the inclusion $Cone(\alpha)\subset Cyl(\alpha)$. Also for the
zero morphism $0:A\rightarrow B$ we have:
\begin{align*}
Cone(0)&=\{a\oplus f\in A\oplus CB | f(1)=0\}=A\oplus C_0((0,1)\rightarrow b)=A\oplus SB\\
Cyl(0)&=\{a\oplus f\in A\oplus \mathbb{I}B | f(1)=0\}=A\oplus
C_0([0,1)\rightarrow B)\simeq A\oplus CB.\end{align*}
\end{remark}
As in the case $Cone(\alpha)$ we have the following proposition
for $Cyl(\alpha)$:
\begin{proposition}\label{6} For a simplicial morphism $\alpha: A_n\rightarrow B_m$ between NCCW
complexes of dimension $n$ and $m$, the $Cyl(\alpha)$ is an NCCW
complex of dimension $max\{n,m+1\}$
\end{proposition}
\begin{proof}
We know that $\mathbb{I}B_m$ is an NCCW complex of dimension
$m+1$. The rest of proof is similar to \ref{2}.
\end{proof}
In the next section we will apply the following fact.
\begin{proposition}\label{4}
If $\alpha:A\rightarrow B$ is a C*-morphism, then $A$ is a
deformation retract of $Cyl(\alpha)$. \end{proposition}
\begin{proof}
Let $\pi: Cyl(\alpha)\rightarrow A$ be given by $\pi(a\oplus
f):=a$ and $\eta: A\rightarrow Cyl(\alpha)$ by
$\eta(a):=a\oplus\alpha(a)$, where $\alpha(a)$ denotes the
constant map $t\mapsto\alpha(a)$. Then $\pi\circ\eta =id_A$. Now
we define the C*-morphism $\psi :
Cyl(\alpha)\rightarrow\mathbb{I}Cyl(\alpha)$ by
\begin{align*}
\psi(a\oplus f):&\mathbb{I}\longrightarrow Cyl(\alpha)\\
&t\longmapsto a\oplus f_t\end{align*} where
$f_t:\mathbb{I}\rightarrow B$ is defined as $f_t(s)=f((1-s)t+s)$
for each $s\in\mathbb{I}$. Then we can see that
\begin{align*}
\psi(a\oplus f)(0)&=a\oplus f=id_{Cyl(\alpha)}(a\oplus f)\\
\psi(a\oplus f)(1)&=a\oplus
f(1)=a\oplus\alpha(a)=\eta\circ\pi(a\oplus f))
\end{align*}
and so $\eta\circ\pi$ and $id_{Cyl(\alpha)}$ are homotopic.
\end{proof}
\section{ NC mapping telescope}
In this section we generalize the notions of NC mapping cylinder
and NC mapping cone. Their related exact sequences are studied,
and their $K_0$ and $K_1$ groups are obtained and the conditions
for their NCCW complex structure are specified.
\begin{definition}
For a sequence of length $n$ of C*-algebras
\begin{equation}
A_1\xrightarrow[{\qquad}]{\alpha_1}A_2\xrightarrow[{\qquad}]{\alpha_2}\dotsb\xrightarrow{\alpha_{n-1}}A_n\xrightarrow[{\qquad}]{\alpha_n}A_{n+1}
\end{equation}
the {\it NC mapping cylindrical telescope} is defined by
\begin{align*}
T_n:=\{&a\oplus f_2\oplus\dots\oplus f_{n+1}\in
A_1\oplus\mathbb{I}A_2\oplus\dots\oplus\mathbb{I}A_{n+1}
|\\
&f_{k+1}(1)=\alpha_k\circ\alpha_{k-1}\circ\dots\circ\alpha_1(a) ,
k=1,2,\dots,n\}
\end{align*}
\end{definition}
Since each $\alpha_k$ is continuous, $T_n$ is a closed subalgebra
of the C*-algebra
$A_1\oplus\mathbb{I}A_2\oplus\dots\oplus\mathbb{I}A_{n+1}$, and so
it is a C*-algebra. Also we note that for $n=1$,
$T_1=Cyl(\alpha_1)$.
\begin{proposition}\label{7}
For each $n\geq 2$, $T_{n-1}$ is a deformation retract of $T_n$.
\end{proposition}
\begin{proof}
The proof is done by induction on $n$. For $n=1$, let $\beta
:T_1\rightarrow A_3$ be defined by $\beta_2(a\oplus
f_2)=\alpha_2\circ\alpha_1(a)$. Then $\beta_2$ is a C*-morphism
and from proposition \ref{4},  $T_1$ is a deformation retract of
$Cyl(\beta_2)$. But $Cyl(\beta_2)=T_2$, since
\begin{align*}
Cyl(\beta_2)&=\{(a\oplus f_2)\oplus f_3\in
T_1\oplus\mathbb{I}A_3|\beta_2(a\oplus f_2)=f_3(1)\}\\
&=\{a\oplus f_2\oplus f_3\in
A_1\oplus\mathbb{I}A_2\oplus\mathbb{I}A_3|f_2(1)=\alpha_1(a) ,
f_3(1)=\alpha_2\circ\alpha_1(a)\}\\
&=T_2.
\end{align*}
Now by induction if we define $\beta_n:T_{n-1}\rightarrow A_{n+1}$
by $\beta_n(a\oplus f_2\oplus\dots\oplus
f_n)=\alpha_n\circ\alpha_{n-1}\circ\dots\circ\alpha_1(a)$, we see
that $T_{n-1}$ is a deformation retract of $T_n$.\end{proof}
\begin{corollary}
For the sequence (1), $A_1$ is a deformation retract of $T_n$.
\end{corollary}
\begin{proof}
From the previous proposition, $T_n$ deformation retract to
$T_1=Cyl(\alpha_1)$ and since by \ref{4},  $Cyl(\alpha_1)$
deformation retracts to $A_1$, so $A_1$ is a deformation retract
of  $T_n$.
\end{proof}
\begin{corollary}
For the sequence (1), $K_0(A_1)=K_0(T_1)=\dots =K_0(T_n)$.
\end{corollary}
\begin{proposition}
For the sequence (1), if
$\alpha_m\circ\alpha_{m-1}\circ\dots\circ\alpha_{m-k}=0$ for some
$k<m\leq n$, then $$T_n\simeq
T_{m-1}\oplus(\bigoplus\limits_{i=m+1}^{n+1}CA_i)$$
\end{proposition}
\begin{proof}
\begin{align*}
T_n&=\{a\oplus f_2\oplus\dots\oplus f_{n+1}\in
A_1\oplus\mathbb{I}A_2\oplus\dots\oplus\mathbb{I}A_{n+1}
|\\
&\quad
f_{k+1}(1)=\alpha_k\circ\alpha_{k-1}\circ\dots\circ\alpha_1(a) ,
k=1,2,\dots,n\}\\
&=\{a\oplus f_2\oplus\dots\oplus f_{n+1}\in
A_1\oplus\mathbb{I}A_2\oplus\dots\oplus\mathbb{I}A_{n+1}
|\\
&\quad
f_{k+1}(1)=\alpha_k\circ\alpha_{k-1}\circ\dots\circ\alpha_1(a),
1\leq k< m; f_{k+1}(1)=0,m\leq k\leq m\}\\
&\simeq T_{m-1}\oplus CA_{m+1}\oplus\dots\oplus CA_{n+1}.
\end{align*}
\end{proof}
\begin{corollary}
For the exact sequence
$A_1\xrightarrow[{}]{\alpha_1}A_2\xrightarrow[{}]{\alpha_2}\dotsb\xrightarrow{\alpha_{n-1}}A_n\xrightarrow[{}]{\alpha_n}A_{n+1}$
of length $n$, $T_n\simeq T_{n-1}\oplus CA_{n+1}$ and in
particular $$T_n\simeq
Cyl(\alpha_1)\oplus(\bigoplus\limits_{i=3}^{n+1}CA_i).$$
\end{corollary}
\begin{proposition}
If
$A_1\xrightarrow[{\qquad}]{\alpha_1}A_2\xrightarrow[{\qquad}]{\alpha_2}\dotsb\xrightarrow{\alpha_{n-1}}A_n\xrightarrow[{\qquad}]{\alpha_n}A_{n+1}$
is a sequence of simplicial morphism between NCCW complexes of
dimensions $m_1,m_2,\dots,m_{n+1}$ then $T_n$ is an NCCW complex
of dimension $max\{m_1,1+m_2,1+m_3,\dots,1+m_{n+1}\}$.
\end{proposition}
\begin{proof} Since $\alpha_1:A_1\rightarrow A_2$
is a simplicial morphism , by \ref{6}, $T_1=Cyl(\alpha_1)$ is an
NCCW complex of dimension $max\{m_1,1+m_2\}$. Let $\beta_2:
T_1\rightarrow A_3$ be the C*-morphisms of proposition \ref{7}.
Since $\alpha_1$ and $\alpha_2$ are simplicial morphism then so is
$\beta_2$, and $T_2=Cyl(\beta_2)$ is an NCCW complex of dimension
$max\{max\{m_1,1+m_2\},1+m_3\}=max\{m_1,1+m_2,1+m_3\}$.
Inductively, the morphism
\begin{align*}
\beta_n:&T_{n-1}\longrightarrow A_{n+1}\\
&a\oplus f_2\oplus\dots\oplus
f_n\mapsto\alpha_n\circ\alpha_{n-1}\circ\dots\circ\alpha_1(a)
\end{align*}
is simplicial and so $T_n=Cyl(\beta_n)$ is an NCCW complex of
dimension $max\{m_1,1+m_2,1+m_3,\dots,1+m_{n+1}\}$.
\end{proof}
\begin{definition}
Let
$A_1\xrightarrow[{\quad}]{\alpha_1}A_2\xrightarrow[{\quad}]{\alpha_2}\dotsb\xrightarrow{\alpha_{n-1}}A_n\xrightarrow[{\quad}]{\alpha_n}A_{n+1}$
be a sequence of C*-morphisms. The {\it NC mapping conical
telescope} is defined as
\begin{align*}
T_nC:=\{&a\oplus f_2\oplus\dots\oplus f_{n+1}\in A_1\oplus
CA_2\oplus\dots\oplus CA_{n+1}
|\\
&f_{k+1}(1)=\alpha_k\circ\alpha_{k-1}\circ\dots\circ\alpha_1(a) ,
k=1,2,\dots,n\}
\end{align*}
\end{definition}
As in the case $T_n$, for $n=1$, $T_1C=Cone(\alpha_1)$.
\begin{proposition}
For the sequence (1), if
$\alpha_m\circ\alpha_{m-1}\circ\dots\circ\alpha_{m-k}=0$ for some
$k<m\leq n$, then
$$T_nC=T_{m-1}C\oplus(\bigoplus\limits_{i=m+1}^{n+1}SA_i)$$
\end{proposition}
\begin{proof}
\begin{align*}
T_nC&=\{a\oplus f_2\oplus\dots\oplus f_{n+1}\in A_1\oplus
CA_2\oplus\dots\oplus CA_{n+1}
|\\
&\quad
f_{k+1}(1)=\alpha_k\circ\alpha_{k-1}\circ\dots\circ\alpha_1(a) ,
k=1,2,\dots,n\}\\
&=\{a\oplus f_2\oplus\dots\oplus f_{n+1}\in A_1\oplus
CA_2\oplus\dots\oplus CA_{n+1}
|\\
&\quad
f_{k+1}(1)=\alpha_k\circ\alpha_{k-1}\circ\dots\circ\alpha_1(a),
1\leq k< m; f_{k+1}(1)=0,m\leq k\leq m\}\\
&= T_{m-1}C\oplus SA_{m+1}\oplus\dots\oplus SA_{n+1}.
\end{align*}
\end{proof}
\begin{corollary}
If the sequence (1) is exact, then $T_nC=T_{n-1}\oplus SA_{n+1}$,
and in particular $$T_nC=
Cone(\alpha_1)\oplus(\bigoplus\limits_{i=3}^{n+1}SA_i).$$
\end{corollary}
\begin{proposition}
For each sequence of length $n>1$,there is an exact sequence
$0\longrightarrow SA_{n+1}\longrightarrow T_nC\longrightarrow
T_{n-1}\longrightarrow 0$. \end{proposition}
\begin{proof}Let $i: SA_{n+1}\rightarrow T_nC$ be
the inclusion morphism $f\mapsto 0\oplus 0\oplus\dots\oplus f$ and
$\pi : T_nC\rightarrow T_{n-1}C$ be the projection $a\oplus
f_2\oplus\dots\oplus f_{n+1}\mapsto a\oplus f_2\oplus\dots\oplus
f_n$, then $ker\pi=i(SA_{n+1})$.
\end{proof}
\begin{proposition}
For each sequence of length $n$, there exists an exact sequence
$$0\longrightarrow T_nC\longrightarrow T_n\longrightarrow\bigoplus\limits_{k=2}^{n+1}A_k\longrightarrow 0.$$
\end{proposition}
\begin{proof}
Let $i:T_nC\rightarrow T_n$ be the obvious inclusion and
$\pi:T_n\rightarrow \bigoplus\limits_{k=2}^{n+1}A_k$ be defined by
$\pi(a\oplus f_2\oplus\dots\oplus
f_{n+1}):=f_2(0)\oplus\dots\oplus f_{n+1}(0)$, Then $ker\pi
=T_nC$.
\end{proof}
\begin{proposition}
for each sequence of length $n$, ther exists a cyclic six term
exact Sequence,
$$\begin{CD}
K_0(T_nC)@>>>K_0(T_n)@>>>\bigoplus\limits_{i=2}^{n+1}K_0(A_i)\\
@AAA&&@VVV\\
\bigoplus\limits_{i=2}^{n+1}K_1(A_i)@<<<K_1(T_n)@<<<K_1(T_nC)\end{CD}$$
\end{proposition}
\begin{proof}
Since $T_nC$ is an ideal in $T_n$, and
$T_n/T_nC=\bigoplus\limits_{i=2}^{n+1}A_i$, and
$K_0(\bigoplus\limits_{i=2}^{n+1}A_i)=\bigoplus\limits_{i=2}^{n+1}K_0(A_i)$
and $
K_1(\bigoplus\limits_{i=2}^{n+1}A_i)=\bigoplus\limits_{i=2}^{n+1}K_1(A_i)$,
the exactness of the six term sequence follows from [W,9.3.2].
\end{proof}
\begin{proposition}
If
$A_1\xrightarrow[{\quad}]{\alpha_1}A_2\xrightarrow[{\quad}]{\alpha_2}\dotsb\xrightarrow{\alpha_{n-1}}A_n\xrightarrow[{\quad}]{\alpha_n}A_{n+1}$
is a sequence of simplicial morphisms between NCCW complexes of
dimensions $m_1,m_2,\dots,m_{n+1}$, then $T_nC$ is an NCCW complex
of dimension $max\{m_1,1+m_2,1+m_3,\dots,1+m_{n+1}\}$.
\end{proposition}
\begin{proof}
Since $\alpha_1:A_1\rightarrow A_2$ is a simpicial morphism, by
\ref{2} , $T_1C=Cone(\alpha_1)$ is an NCCW complex of dimension
$max\{m_1,1+m_2\}$. Let $\beta_2: T_1C\rightarrow A_3$ be defined
as $\beta(a\oplus f_2)=\alpha_2\circ\alpha_1(a)$. Since $\alpha_1$
and $\alpha_2$ are simplicial morphism then so is $\beta_2$, and
$T_2C=Cone(\beta_2)$ is an NCCW complex of dimension
$max\{max\{m_1,1+m_2\},1+m_3\}=max\{m_1,1+m_2,1+m_3\}$.
Inductively, the morphism
\begin{align*}
\beta_n:&T_{n-1}C\longrightarrow A_{n+1}\\
&a\oplus f_2\oplus\dots\oplus
f_n\mapsto\alpha_n\circ\alpha_{n-1}\circ\dots\circ\alpha_1(a)
\end{align*}
is simplicial and so $T_nC=Cone(\beta_n)$ is an NCCW complex of
dimension $max\{m_1,1+m_2,1+m_3,\dots,1+m_{n+1}\}$.\end{proof}

\end{document}